\documentclass{amsart}

\usepackage{amsmath, graphicx, cite, enumerate, mathrsfs, comment, color}
\usepackage{amssymb, amsthm, amscd, cancel}
\usepackage[paper=a4paper,headheight=0pt,left=3cm,top=4cm,right=3cm,bottom=4cm]{geometry}

\newtheorem{theorem}{Theorem}[section]

\newtheorem{proposition}[theorem]{Proposition}
\newtheorem{corollary}[theorem]{Corollary}

\theoremstyle{definition}
\newtheorem{definition}[theorem]{Definition}

\theoremstyle{remark}
\newtheorem{remark}[theorem]{Remark}

\numberwithin{equation}{section}

\DeclareMathAlphabet{\mathpzc}{OT1}{pzc}{m}{it}

\newcommand{\abs}[1]{\left|#1\right|}

\newcommand{\Rm}{\textup{Rm}}
\newcommand{\Ric}{\textup{Ric}}

\newcommand{\id}{\textup{id}}

\newcommand{\R}{\mathbb{R}}

\newcommand{\CP}{\mathbb{CP}}

\renewcommand{\div}{\,\textup{div}}

\newcommand{\D}[2]{\frac{\partial #1}{\partial #2}}

\begin{document}
\title[]{Uniqueness and Symmetry of Self-Similar Solutions of Curvature Flows in Warped Product Spaces}
\author{Frederick Tsz-Ho Fong}
\address{Department of Mathematics, School of Science, Hong Kong University of Science and Technology, Clear Water Bay, Hong Kong}
\email{frederick.fong@ust.hk}
\date{12 November, 2024}
\begin{abstract}
In this article, we establish some uniqueness and symmetry results of self-similar solutions to curvature flows by some homogeneous speed functions of principal curvatures in some warped product spaces. In particular, we proved that any compact star-shaped self-similar solution to any parabolic flow with homogeneous degree $-1$ (including the inverse mean curvature flow) in warped product spaces $I \times_{\phi} M^n$, where $M^n$ is a compact homogeneous manifold and $\phi'' \geq 0$, must be a slice. The same result holds for compact self-expanders when the degree of the speed function is greater than $-1$ and with an extra assumption $\phi' \geq 0$.

Furthermore, we also show that any complete non-compact star-shaped, asymptotically concial expanding self-similar solutions to the flow by positive power of mean curvature in hyperbolic and anti-deSitter-Schwarzschild spaces are rotationally symmetric.
\end{abstract}

\maketitle

\section{Introduction}
\subsection{Definitions and background}
In this article, we consider extrinsic curvature flows in warped product ambient spaces $(N^{n+1},\bar{g})$ of the form
\begin{align}
\label{eq:warp}
N^{n+1} & = I \times M^n, &\bar{g} & = d\rho^2 + \phi(\rho)^2 g_M	
\end{align}
where $I$ is an interval on the real line, $(M^n, g_M)$ is an $n$-dimensional Riemannian manifold, and $\phi(\rho)$ is a warp function such that $\phi(\rho) \geq 0$ on $I$.

When taking $I$, $M^n$, $\phi$ to be the following, we get constant curvature space forms:
\begin{itemize}
\item Euclidean space $\R^{n+1}$: $I = [0,\infty)$, $M^n = \mathbb{S}^n$, and $\phi(\rho) = \rho$.
\item Hyperbolic space $\mathbb{H}^{n+1}$: $I = [0,\infty)$, $M = \mathbb{S}^n$, and $\phi(\rho) = \sinh \rho$.
\item Sphere $\mathbb{S}^{n+1}$: $I = [0,\pi]$, $M = \mathbb{S}^n$, and $\phi(\rho) = \sin \rho$.
\end{itemize}
One can check that then $\bar{g}$ has constant sectional curvature $0,-1,1$ respectively.

There are other Riemannian manifolds of interest in the context of general relativity, such as the anti-deSitter-Schwarzschild manifold $[s_0,\infty) \times \mathbb{S}^n$ with the metric:
\[\frac{1}{1-ms^{1-n}+s^2}ds^2 + s^2 g_{\mathbb{S}^n}\]
where $g_{\mathbb{S}^n}$ is the metric of a round unit sphere. After changing variables to $\rho$ via 
\[\dfrac{d\rho}{ds} = \dfrac{1}{\sqrt{1-ms^{1-n}+s^2}},\]
it can be written in the warped product form as in \eqref{eq:warp}, with the warped product factor $\phi(\rho)$ asymptotically approaching $\sinh\rho$ as $\rho \to \infty$. Other similar examples include Reissner-Nordstrom metric, and deSitter-Schwarzschild metric.

Now back to the general setting as in \eqref{eq:warp}, we consider the vector field
\[X := \phi(\rho)\dfrac{\partial}{\partial\rho}\]
on $N^{n+1}$. When $N^{n+1} = \R^{n+1} = [0,\infty) \times \mathbb{S}^n$ with $\phi(\rho) = \rho$, the vector field $X$ is the position vector. Generally, $X$ is a conformal Killing vector field, as we have
\begin{equation}
\mathcal{L}_X \bar{g} = 2\phi'(\rho)\bar{g}.
\end{equation}
Therefore, one can regard $X$ in $N^{n+1}$ as a natural generalization of the position vector field. A hypersurface $\Sigma$ is said to be star-shaped if $\bar{g}(X,\nu) \not= 0$ on $\Sigma$. Here, $\nu$ is the Gauss map of $\Sigma$.

Let $\psi^X_t : N \to N$ be the 1-parameter family of diffeomorphisms generated by $X$, i.e.
\[\D{}{t}\psi^X_t = X \circ \psi^X_t.\]
By expressing $\psi_t^X$ as the map $\psi_t^X(\rho,w) = (r(\rho,t), w)$, then one can see that $r(t)$ is the solution to the ODE:
\begin{equation}
\label{eq:r(t)}
\D{}{t}r(\rho, t) = \phi\big(r(\rho, t)\big), \quad r(\rho, 0) = \rho.	
\end{equation}

In case $N^{n+1}$ is a constant curvature space form with sectional curvature $k = -1, 0, 1$, it can be computed that for any $(\rho, w) \in I \times M$, we have
\begin{equation}
\psi_t^X(\rho, w) =
\begin{cases}
\left(2\tanh^{-1}(e^t \tanh\frac{\rho}{2}),w\right) & \text{ if $k = -1$}\\
(e^t\rho, w) & \text{ if $k = 0$}\\
\left(2\tan^{-1}(e^t \tan\frac{\rho}{2}),w\right)	& \text{ if $k = +1$}
\end{cases}.
\end{equation}

We say an evolving hypersurface $\Sigma^n_t$ to be a self-similar solution in $(N^{n+1}, \bar{g})$ if there exists a strictly monotone time function $\tau(t)$ of $t$ such that $\Sigma_t^n = \psi_{\tau(t)}^X\left(\Sigma_0\right)$. If $\tau'(t) > 0$, we call $\Sigma_t$ an \emph{expanding self-similar solution} or (simply \emph{self-expanders}), whereas if $\tau'(t) < 0$, we call $\Sigma_t$ a \emph{shrinking self-similar solution} (or simply \emph{self-shrinkers}). If such a self-similar solution is also a solution along the flow
\begin{equation}
\label{eq:flow-perp}
\left(\D{F_t}{t}\right)^\perp = f\nu_t
\end{equation}
where $F_t$ is the embedding of $\Sigma_t$ in $N$ with the Gauss map $\nu_t$, $\perp$ denotes the projection onto the normal direction, and $f(\lambda_1, \cdots, \lambda_n)$ is a homogeneous symmetric function of principal curvatures of $\Sigma_t$, then we call $\Sigma_t$ a self-similar solution to that flow in $(N^{n+1}, \bar{g})$.

Slices of \eqref{eq:warp} are important examples of self-similar solutions. Consider a slice $\Sigma^n(\rho_0) := \{\rho_0\} \times M^n$ in $N^{n+1} = I \times M^n$, where $\rho_0 \in I$. One can parametrize $\Sigma(\rho_0)$ by
\[F(w_1, \cdots, w_n) = (\rho_0, w_1, \cdots, w_n)\]
where $w_i$'s are local coordinates of $M^n$. Then, the tangent space of $\Sigma(\rho_0)$ is spanned by 
\[\left\{\D{F}{w_i}\right\}_{i=1}^n=\left\{\D{}{w_i}\right\}_{i=1}^n.\]
Take the Gauss map $\nu$ to be the inward-pointing unit normal $\displaystyle{\nu = -\D{}{\rho}}$. This shows the principal curvatures of $\Sigma(\rho_0)$ are given by
\[\frac{\phi'(\rho_0)}{\phi(\rho_0)}, \cdots, \frac{\phi'(\rho_0)}{\phi(\rho_0)}.\]
In case of constant curvature space forms with $\textup{sect} = k$, they are given by
\[
\frac{\phi'(\rho_0)}{\phi(\rho_0)} =
\begin{cases}
\frac{\cosh \rho_0}{\sinh \rho_0} & \text{ if $k = -1$}\\
\frac{1}{r_0} & \text{ if $k = 0$}\\
\frac{\cos\rho_0}{\sin\rho_0} & \text{ if $k = +1$}	
\end{cases}.
\]
Under the flow \eqref{eq:flow-perp}, a slice $\Sigma_0 = \{\rho_0\} \times M^n$ would remain to be slices $\Sigma_t = \{\rho(t)\} \times M^n$ with $\rho(t)$ given by the solution to be the ODE:
\[\rho'(t) = -f\left(\frac{\phi'(\rho)}{\phi(\rho)}, \cdots, \frac{\phi'(\rho)}{\phi(\rho)}\right).\]
Therefore, any slice $\{\rho_0\} \times M^n$ in $N^{n+1}$ gives a self-similar solution $\Sigma_t = \{\rho(t)\} \times M^n$ along any flow of the form \eqref{eq:flow-perp} as long as principal curvatures are in the domain of $f$. In terms of the flow map $\psi_t^X$, the self-similar solution can be written as
\[\Sigma_t = \psi_{\tau(t)}^X \circ \Sigma_0\]
where $\tau(t)$ is a function satisfying $\rho(t) = r(\rho_0, \tau(t))$ where $r(\rho, t)$ is the solution to \eqref{eq:r(t)}.

\subsection{Compact self-similar solutions}
One objective of this article is to prove that under some general conditions on the warp function $\phi(\rho)$, the warp factor $M^n$, and the speed function $f$, these slices are the only compact star-shaped self-similar solutions. We also establish some rotational symmetry results for complete non-compact self-similar solutions assuming it is asymptotic to a ``cone'' in some doubly-warped product space.

Classification and uniqueness problems of self-similar solutions have been a central topic in the context of singularity analysis of geometric flows, as they are likely be limit models of the flow. When the ambient space is $\R^{n+1}$, Huisken proved in \cite{H84} that closed convex hypersurface in $\R^{n+1}$ would converge to the round sphere after rescaling under the mean curvature flow (MCF), and in \cite{H90} he proved that compact mean-convex MCF self-shrinkers in $\R^{n+1}$ must be the round sphere. Colding-Minicozzi \cite{CM12} proved that stable MCF self-shrinkers must be the round sphere. In $\R^3$, Brendle proved in \cite{B16} that any compact genus $0$ self-shrinkers in $\R^3$ must be the round $2$-sphere. When the speed function is given by some homogeneous function of positive degrees, there are some uniqueness results by McCoy \cite{McC11} in Euclidean spaces. Recently, there are some breakthroughs on the Gauss curvature (and its power) flow (GCF) by Brendle-Choi-Daskalopoulos \cite{BCD17}, proving that the compact self-shrinkers to the flow by $K^\alpha\nu$ must be either spheres or ellipsoids depending on $\alpha > 0$. There are also generalizations of similar uniqueness results of compact self-shrinkers for $\sigma_k^\alpha$-flows with $\alpha > 0$ in $\R^{n+1}$ by Gao-Li-Ma \cite{GLM18}, and to flows by high positive power of $1$-homogeneous, inverse concave speed function of principal curvatures in $\R^{n+1}$ by Gao-Li-Wang \cite{GLW22}.

For inverse curvature flows which are of expansion type, round spheres in $\R^{n+1}$ are self-expanders instead. The flow was first studied by Gerhardt \cite{G90} and Urbas \cite{U90} who independently proved that inverse curvature flow on star-shaped hypersurfaces by homogeneous speed functions of degree $-1$ in $\R^{n+1}$ converges to the round sphere after rescaling. As for classification results of self-similar solutions, Drugan-Lee-Wheeler proved in \cite{DLW16} proved that compact self-expander to the inverse mean curvature flow (IMCF) in $\R^{n+1}$ must be round spheres using Hsiung-Minkowski's identities. The same uniqueness theorem for flows by $(\sigma_k/\sigma_j)^{1/(j-k)}$ in $\R^{n+1}$, where $\sigma_k$'s are $k$-th symmetric polynomials of principal curvatures, was also proved by Kwong-Lee-Pyo in \cite{KLP18}. In the joint work \cite{CCF21} by A. Chow, K. Chow and the author, we further extend this uniqueness result to any expansion flow by speed functions of degree $-1$ in $\R^{n+1}$.

The Euclidean space $\R^{n+1}$ is a warped product space $[0,\infty) \times \mathbb{S}^n$, and the metric can be expressed as $d\rho^2 + \rho^2 g_{\mathbb{S}^n}$. Round spheres then can be regarded as a ``slice'' $\{\rho_0\} \times \mathbb{S}^n$ of this warped product. Constant curvature space-forms are examples of warped product spaces too. A variant of mean curvature flow was studied by Guan-Li \cite{GL15} when the ambient space is a space-form, and convergence to a slice was proved. The study of inverse curvature flows in non-Euclidean ambient spaces has been driven by the study of geometric inequalities, such as the Riemannian Penrose inequality by Huisken-Ilmanen \cite{HI01} and the Minkowski's inequality in anti-deSitter-Schwarzchild space by Brendle-Hung-Wang \cite{BHW16}, both using the IMCF. There are also works on more general warped product spaces by Zhou \cite{Z18} on IMCF, and Scheuer \cite{S17,S19} on more general homogeneous speed functions, where long-time convergence to an umbilical limit was proved under some natural assumptions on the warping function.

This article is mainly about establishing uniqueness and symmetry results of curvature flows in warped product spaces. In this regard, there are works about contraction flows by Gao-Ma \cite{GM19}, who proved the uniqueness of self-similar solutions to the $\sigma_k^\alpha$-flow and $S_k^\alpha$-flow (where $S_k = \sum_i \lambda_i^k$) with $\alpha > 0$ in the hemisphere, on 3-hyperbolic space. For expansion flows, there are recent works by Gao-Ma \cite{GM21} and Gao \cite{G23} which establish uniqueness results of self-similar solutions to flows by negative powers of $1$-homogeneous functions $f(\lambda_1,\cdots,\lambda_n)$ of principal curvatures on warped product spaces $I \times M^n$ with the Ricci curvature of $M^n$ and the warping function $\phi$ satisfy some conditions similar to those appeared in Brendle's work \cite{B13} and Brendle-Eichmair's work \cite{BE13} on constant mean curvature hypersurfaces in warped product spaces, together with some assumptions on the speed function $f$.

In this article, we consider self-similar solutions on warped product spaces $I \times M^n$ with different assumptions on $M^n$ and the speed function $f$. In most cases, we study the case when $M^n$ is a compact Riemannian homogeneous space, whose Lie algebra of Killing fields act transitively on $M^n$ by isometries. Examples of such spaces include the space forms $\R^n$, $\mathbb{S}^n$ and $\mathbb{H}^n$ with constant sectional curvature, and also complex projective space $\CP^n$ with the Fubini-Study metric which is Einstein but the sectional curvatures are not constant. See \cite{BN20} for more discussions on Riemannian homogeneous spaces. On these ambient spaces, we have established the following uniqueness result:

\begin{theorem}[= Theorem \ref{thm:symmetry}, Corollary \ref{cor:uniqueness} and Corollary \ref{cor:uniqueness_deg>-1}]
\label{thm:main1}
Let $N^{n+1} = I \times M^n$ be a warped product space with metric $\bar{g} = d\rho^2 + \phi(\rho)^2 g_M$, and $(M^n,g_M)$ being a compact homogeneous space. Suppose $\Sigma_0^n$ is a compact star-shaped self-similar solution in  along $X = \phi(\rho)\D{}{\rho}$ to the flow \eqref{eq:flow-perp}, i.e.
\[f=\varepsilon\bar{g}(X,\nu)\]
where $\varepsilon$ is a non-zero constant, and $f$ is a homogeneous function of principal curvatures with degree $\deg f$ satisfying $\D{f}{\lambda_i} > 0$ for any $i$. Suppose $\Sigma_0$, $\phi$, $\varepsilon$, and $f$ satisfy the condition
\begin{equation}
\label{eq:my_condition}
\dot{f}^{ij}g_{ij}\frac{\phi''}{\phi} + \varepsilon(1+\deg f)\phi' \geq 0,
\end{equation}
then $\Sigma_0$ must be a slice $\{\rho_0\} \times M^n$. In particular, if $\deg f = -1$, then the above condition holds when $\phi'' \geq 0$. If $\deg f > -1$, then the above condition holds when both $\phi' \geq 0$, $\phi'' \geq 0$, and $\varepsilon > 0$ (i.e. self-expanders).
\end{theorem}
Examples of spaces satisfying both $\phi' > 0$ and $\phi'' > 0$ include the Euclidean space $\R^{n+1}$ (where $\phi = \rho$), the hyperbolic space $\mathbb{H}^{n+1}$ (where $\phi = \sinh\rho$), and more generally the anti-deSitter-Schwarzschild space.

It is interesting to compare the above result with those appeared in Gao-Ma's work \cite{GM21} and Gao's work \cite{G23}. In their works, the factor $(M^n, g_M)$ in the warped product, and the warping function $\phi$ are assumed to satisfy the condition such as $\Ric_M \geq (n-1)((\phi')^2 - \phi\phi'')g_M$ in case of $H^{-\alpha}$ flow (where $\alpha > 0$), or in the case $(M^n, g_M)$ has constant sectional curvature $c \in \R$, the warping function satisfies
\begin{equation}
\label{eq:Gao_condition}	
\phi' > 0 \quad \text{ and } \quad \frac{\phi''}{\phi} + \frac{c-(\phi')^2}{\phi^2} \geq 0.
\end{equation}

In the later case, the speed function is the negative power of a $1$-homogeneous function satisfying some conditions (see Condition 3 in \cite{G23} for detail). The value of $\alpha$ in \cite{GM21,G23} could be any positive number.

In our Theorem \ref{thm:main1}, the factor $M^n$ needs not have constant sectional curvature or some lower bound of Ricci curvature. Instead, we allow it to be any compact homogeneous space whose isometries generated by Killing fields act transitively on $M$. For the warping function $\phi$ requires different conditions than those in \cite{GM21,G23}, as ours require $\phi'' \geq 0$ when $\deg f = -1$, and for $\deg f > -1$ we need $\phi' \geq 0$ and $\phi'' \geq 0$. Euclidean and hyperbolic spaces satisfy both \eqref{eq:Gao_condition} and our condition \eqref{eq:my_condition}, while there are some spaces satisfy only either one. For the speed function $f$, Theorem \ref{thm:main1} basically allow it to be any homogeneous function of degree $> -1$, and that the flow is parabolic in a sense that $\D{f}{\lambda_i} > 0$, whereas in \cite{GM21, G23} there are some other conditions that $f$ needs to fulfill. The methodology adopted to prove the uniqueness results are also different. Key ingredients in \cite{GM21, G23} include the use of Heintze-Karcher inequality, and some maximum principle argument applied to show auxiliary function. We used a different approach by considering the Killing vector fields lifted from those on $M^n$ -- there are plenty of them as a homogeneous space -- and proved that those Killing fields are tangential to the self-similar solution. As the Killing fields span the tangent space of $M^n$, this would show $\Sigma^n$ and $M^n$ have the same tangent subspace in $TN^{n+1}$. The warped product structure then shows $\Sigma^n$ must be a slice. Similar idea are applied to prove the uniqueness of self-similar solution to flows by degree $-1$ homogeneous speed function in $\R^{n+1}$ in \cite{CCF21}, yet there are subtle differences in case of more general warped product spaces. For example, self-similar solutions in $\R^{n+1}$ are homothetic whereas in warped products the self-similar solutions flow in a conformal way.

\subsection{Non-compact self-similar solutions}
Apart from the uniqueness result of compact self-similar solutions, we have also extended some rotational symmetry results of MCF (and its positive power) self-expanders in Euclidean spaces to more general warped product space. In the work \cite{FM19} by the author and McGrath, we proved that mean-convex asymptotically conical MCF self-expanders are rotationally symmetric. Inspired by the fact that $\R^{n+1}$ is a doubly-warped product $[0,\infty) \times [0,\pi] \times \mathbb{S}^{n-1}$ with metric
\[\bar{g} = d\rho^2 + \rho^2 (d\theta^2 + (\sin^2\theta)g_{\mathbb{S}^{n-1}}),\]
and a cone is simply the subset $\{\theta = \theta_0\}$, one can also consider doubly-wrapped product $N^{n+1} = I \times J \times S$ with metric
\[\bar{g} = d\rho^2 + \phi(\rho)^2 (d\theta^2 + r(\theta)^2 g_S)\]
so that a ``cone'' is the subset $\mathcal{C}(\theta_0) = \{\theta = \theta_0\}$ where $\theta_0 \in J$ is fixed.

We have extended our work \cite{FM19} to more general doubly-warped spaces under some conditions on the speed function $f$ and the warping function $\phi$ (see Proposition \ref{prop:non-compact} for detail). In particular, as an application of Proposition \ref{prop:non-compact}, we established the following:

\begin{theorem}[= Theorem \ref{thm:AC}] 
\label{thm:main2}
Suppose $\Sigma_0^n$ is a complete non-compact star-sharped expanding self-similar solution to the $\alpha$-mean curvature flow $\D{F}{t} = H^\alpha\nu$, $\alpha > 0$, on the anti-deSitter-Schwarzschild's space (including hyperbolic space) $N^{n+1} = [s_0,\infty) \times [0,\pi] \times \mathbb{S}^{n-1}$. Assume $\Sigma_0$ is $C^2$-asymptotic to a cone $\mathcal{C}(\theta_0)$, then $\Sigma_0$ must be rotationally symmetric, in a sense that $\Sigma_0$ is invariant under the flow map $\psi_t^K$ generated by any Killing vector field of $\mathbb{S}^{n-1}$.
\end{theorem}

\section{Conformal Killing Fields and Support Functions}
In this section, we consider the support function of a conformal Killing vector field on a hypersurface $\Sigma$ in $N^{n+1}$, and its evolution equation when $\Sigma$ under a flow. A vector field $K$ on $(N^{n+1},\bar{g})$ is said to be a conformal Killing field if there exists a scalar function $\kappa$ such that
\[\mathcal{L}_K\bar{g} = 2\kappa\bar{g}\]
Since the Lie derivative can be locally written as
\[(\mathcal{L}_V\bar{g})_{\alpha\beta} = \bar\nabla_\alpha K_\beta + \bar\nabla_\beta K_\alpha\]
where $\bar\nabla$ is the Levi-Civita connection of $\bar{g}$. By contracting with $\bar{g}^{\alpha\beta}$, we must have $\kappa = \dfrac{\div_{\bar{g}}K}{n+1}$.

Let $\Sigma_t$ be an evolving hypersurface in $N$ parametrized by $F_t(u_1, \cdots, u_n)$ such that
\begin{equation}
\label{eq:f-flow}
\D{F_t}{t} = f\nu_t
\end{equation}
where $f(\lambda_1,\cdots,\lambda_n)$ is a symmetric speed function depending on the principal curvatures $\lambda_i$'s of $\Sigma_t$, and $\nu_t$ is the Gauss map of $\Sigma_t$. Our next goal is to derive the evolution equations of the support function of $K$, defined by
\[\bar{g}( K \circ F_t, \nu_t), \qquad \text{ or in short:} \quad \bar{g}(K,\nu).\]
Some similar results appeared in the author's work \cite{F20} for Killing fields $\R^{n+1}$, in \cite{GL15,GLW19} when $K = X$, and in \cite{FR23} for some specific flow.

Let $\{e_i\}$ be an orthonormal frame on $\Sigma_t$, and regard $f$ as a function of the second fundamental form $h_{ij} = h(e_i,e_j)$, we denote
\[\dot{f}^{ij} := \D{f}{h_{ij}}.\]

\begin{proposition}
Let $(N^{n+1},\bar{g})$ be any Riemannian manifold. Suppose $K$ is a conformal Killing field on $(N^{n+1},\bar{g})$ such that
\[\mathcal{L}_K\bar{g} = 2\kappa\bar{g}\]
for some scalar function $\kappa$. Then, if a hypersurface $\Sigma_t$ evolves by the flow \eqref{eq:f-flow}, we have
\begin{align}
\label{eq:support}
& \left(\D{}{t} - \dot{f}^{ij}\nabla_i\nabla_j\right)\bar{g}(K,\nu)\\
& = \dot{f}^{ij}g_{ij}\bar\nabla_\nu\kappa -\dot{f}^{ij}\overline{\Rm}(e_i,\nu,e_j,\nu)\bar{g}(K,\nu) + f\kappa  + \dot{f}^{ij}h_{ij}\kappa + \dot{f}^{ij} (h^2)_{ij}\bar{g}(K,\nu) \nonumber
\end{align}
\end{proposition}

\begin{proof}
We follow a similar approach as in the author's work \cite{F20}. First let $(x_\alpha)$ be the local coordinates of $(N,\bar{g})$ and $(u_i)$ be the local coordinates of $(\Sigma_t,g_t)$. Then, one can write
\begin{align*}
K & = K^\alpha \D{}{x_\alpha}, & \nu & = \nu^\alpha \D{}{x_\alpha},
\end{align*}
and so locally we have
\[\bar{g}(K,\nu) = \bar{g}_{\alpha\beta}K^\alpha\nu^\beta.\]
Express $F(u_1,\cdots,u_n) = (x_1(u_i), \cdots, x_{n+1}(u_i))$, then \eqref{eq:f-flow} is equivalent to
\begin{equation}
\D{x_\alpha}{t} = f\nu^\alpha
\end{equation}
By the chain rule, we have:
\begin{align*}
\D{}{t}\bar{g}(K,\nu) & = \D{\bar{g}_{\alpha\beta}}{x_\gamma}\D{x_\gamma}{t} K^\alpha \nu^\beta + \bar{g}_{\alpha\beta} \D{K^\alpha}{x_\gamma}\D{x_\gamma}{t} \nu^\beta + \bar{g}_{\alpha\beta}K^\alpha \D{\nu^\beta}{t}\\
& = \partial_\gamma \bar{g}_{\alpha\beta} \cdot f\nu^\gamma \nu^\beta K^\alpha + f\bar{g}_{\alpha\beta}\D{K^\alpha}{x_\gamma}\nu^\gamma\nu^\beta - \bar{g}(K,\nabla f).
\end{align*}
Here we used the fact that
\[\D{\nu}{t} = -\nabla f.\]
Then under normal coordinates $\bar{g}_{\alpha\beta} = \delta_{\alpha\beta}$ and $\partial_\gamma \bar{g}_{\alpha\beta} = 0$, we have
\[\D{}{t}\bar{g}(K,\nu) = f\D{K^\alpha}{x_\gamma}\nu^\gamma\nu^\alpha - \bar{g}(K, \nabla f)\]
For the first term, we consider:
\[\left(\D{K^\alpha}{x_\gamma} + \D{K^\gamma}{x_\alpha}\right)\nu^\gamma\nu^\alpha = \left(\mathcal{L}_K \bar{g}\right)_{\alpha\gamma} \nu^\alpha\nu^\gamma = 2\kappa,\]
and hence we have
\[\D{K^\alpha}{x_\gamma}\nu^\gamma\nu^\alpha = \kappa.\]
This concludes that
\begin{equation}
\label{eq:ddt}
\D{}{t}\bar{g}(K,\nu) = f\kappa - \bar{g}(K,\nabla f).
\end{equation}

Next we compute the Hessian term $\nabla_i\nabla_j \bar{g}(K,\nu)$, using orthonormal frame computations. Let $\{e_i\}_{i=1}^n$ be an orthonormal frame of $T\Sigma$ such that at a point $p \in \Sigma$ we have $\nabla_{e_i}e_j = 0$. Denote $\nabla_i := \nabla_{e_j}$ to be the covariant derivative on $\Sigma$, and $\bar\nabla$ to be the Levi-Civita connection of $(N,\bar{g})$. Recall that $K$ is a conformal Killing field, so we have
\[\bar{g}(\bar\nabla_X K, Y) + \bar{g}(\bar\nabla_Y K, X) = 2\kappa\bar{g}(X,Y)\]
for any vector fields $X, Y$ on $N$. Using this, we can derive that at $p$, we have
\begin{align}
\label{eq:Hess}
& \nabla_i\nabla_j\big(\bar{g}(K,\nu))\\
& = e_i\big(e_j(\bar{g}(K,\nu))\big)-(\nabla_{e_i}e_j)\bar{g}(K,\nu) \nonumber \\
& = e_i\big(\bar{g}(\bar\nabla_j K, \nu) + \bar{g}(K, \bar\nabla_j\nu)\big) - 0 \nonumber\\
& = e_i\big(2\kappa\bar{g}(e_j,\nu)-\bar{g}(\bar\nabla_\nu K, e_j)\big)+\bar{g}(\bar\nabla_i K, \bar\nabla_j\nu)+\bar{g}(K,\bar\nabla_i\bar\nabla_j\nu) \nonumber\\
& = 0 - \bar{g}(\bar\nabla_i\bar\nabla_\nu K, e_j) - \bar{g}(\bar\nabla_\nu K, \bar\nabla_{e_i}e_j) +\bar{g}(\bar\nabla_i K, \bar\nabla_j\nu)+\bar{g}(K,\bar\nabla_i\bar\nabla_j\nu) \nonumber\\
& = -\underbrace{\bar{g}(\bar\nabla_i\bar\nabla_\nu K, e_j)}_{\text{(I)}} - \underbrace{\bar{g}(\bar\nabla_\nu K, h_{ij}\nu)}_{\text{(II)}} + \underbrace{\bar{g}(\bar\nabla_i K, \bar\nabla_j\nu)}_{\text{(III)}} + \underbrace{\bar{g}(K,\bar\nabla_i\bar\nabla_j\nu)}_{\text{(IV)}}. \nonumber
\end{align}
For term (I), we consider the Riemann curvature tensor $\overline{\Rm}$ of $(N, \bar{g})$:
\begin{equation}
\label{eq:term_I}
\overline{\Rm}(e_i,\nu,K,e_j) = \bar{g}(\bar\nabla_i\bar\nabla_\nu K-\bar\nabla_\nu\bar\nabla_i K - \bar\nabla_{[e_i,\nu]}K, e_j).	
\end{equation}
For term (II) of \eqref{eq:Hess}, we recall that $K$ is a conformal Killing field, so
\begin{equation}
\label{eq:term_II}
\bar{g}(\bar\nabla_\nu K, \nu) + \bar{g}(\bar\nabla_\nu K, \nu) = 2\kappa \bar{g}(\nu,\nu) = 2\kappa \implies \bar{g}(\bar\nabla_\nu K, \nu) = \kappa.	
\end{equation}
Note that $\bar\nabla_j\nu = -h_{jl}e_l$, and so
\[\bar\nabla_i \bar\nabla_j \nu = -(\nabla_i h_{jl})e_l - h_{jl}\bar\nabla_i e_l = -(\nabla_ih_{jl})e_l - h_{jl}h_{il}\nu.\]
The term (IV) of \eqref{eq:Hess} can then be written as
\begin{equation}
\label{eq:term_IV}
\bar{g}(K,\bar\nabla_i\bar\nabla_j\nu) = -(\nabla_i h_{jl})K_l - h_{jl}h_{il}\bar{g}(K,\nu).
\end{equation}
Substituting \eqref{eq:term_I}, \eqref{eq:term_II} and \eqref{eq:term_IV} back into \eqref{eq:Hess}, we then get:
\begin{align*}
& \nabla_i\nabla_j\big(\bar{g}(K,\nu))\\
& = -\overline{\Rm}(e_i,\nu,K,e_j)-\bar{g}(\bar\nabla_\nu\bar\nabla_iK + \bar\nabla_{[e_i,\nu]}K, e_j)\\
& \hskip 0.5cm -  h_{ij}\kappa + \bar{g}(\bar\nabla_i K, \bar\nabla_j \nu) - (\nabla_i h_{jl})K_l - h_{il}h_{lj}\bar{g}(K,\nu)\\
& = -\overline{\Rm}(e_i,\nu,K,e_j) - \nu\big(\bar{g}(\bar\nabla_i K,e_j)\big) + \bar{g}(\bar\nabla_iK,\bar\nabla_\nu e_j) - \bar{g}(\bar\nabla_{[e_i,\nu]}K,e_j)\\
& \hskip 0.5cm -  h_{ij}\kappa  + \bar{g}(\bar\nabla_i K, \bar\nabla_j \nu) - (\nabla_i h_{jl})K_l - h_{il}h_{lj}\bar{g}(K,\nu).
\end{align*}
Next we take the trace by $\dot{f}^{ij}$ which is symmetric, so by swapping some $e_i$ and $e_j$, and by writing
\[\dot{f}^{ij}\nu\big(\bar{g}(\bar\nabla_i K, e_j)\big) = \frac{1}{2}\dot{f}^{ij}\nu\big(\bar{g}(\bar\nabla_i K, e_j) + \bar{g}(\bar\nabla_j K, e_i)\big),\]
we get:

\begin{align*}
& \dot{f}^{ij}\nabla_i\nabla_j \big(\bar{g}(K,\nu)\big)\\
& = -\dot{f}^{ij}\overline{\Rm}(e_i,\nu,K,e_j) - \frac{1}{2}\dot{f}^{ij}\nu\big(\bar{g}(\bar\nabla_i K, e_j) + \bar{g}(\bar\nabla_j K, e_i)\big)\\
& \hskip 0.5cm +\dot{f}^{ij}\bar{g}(\bar\nabla_jK,\bar\nabla_\nu e_i + \bar\nabla_i\nu) - \dot{f}^{ij}\bar{g}(\bar\nabla_{[e_i,\nu]}K,e_j)\\
& \hskip 0.5cm  - \dot{f}^{ij}h_{ij}\kappa  - \dot{f}^{ij}(\nabla_i h_{jl})K_l - \dot{f}^{ij} (h^2)_{ij}\bar{g}(K,\nu)
\end{align*}
Using the fact that $K$ is a conformal Killing field, so that
\[\bar{g}(\bar\nabla_i K, e_j) + \bar{g}(\bar\nabla_j K, e_i) = 2\kappa\bar{g}(e_i,e_j),\]
and recalling that
\[[\nu, e_i] = \bar\nabla_\nu e_i - \bar\nabla_i\nu,\]
we can proceed our calculation and show: 
\begin{align}
\label{eq:Hess_2}
& \dot{f}^{ij}\nabla_i\nabla_j \big(\bar{g}(K,\nu)\big)\\
& = \dot{f}^{ij}\overline{\Rm}(e_i,\nu,e_j,K) - \dot{f}^{ij}\nu(\kappa \bar{g}(e_i,e_j)) + \dot{f}^{ij}\big(\bar{g}(\bar\nabla_jK, [\nu,e_i]) + \bar{g}(\bar\nabla_{[\nu,e_i]}K,e_j)\big)\nonumber \\
& \hskip 0.5cm  + 2\dot{f}^{ij}\bar{g}(\bar\nabla_jK,\bar\nabla_i\nu) - \dot{f}^{ij}h_{ij}\kappa  - \dot{f}^{ij}(\nabla_i h_{jl})K_l - \dot{f}^{ij} (h^2)_{ij}\bar{g}(K,\nu) \nonumber\\
& = \dot{f}^{ij}\overline{\Rm}(e_i,\nu,e_j,K) - \dot{f}^{ij}g_{ij}\bar\nabla_\nu\kappa + \underbrace{2\dot{f}^{ij}\kappa\bar{g}([\nu,e_i],e_j)}_{\text{(A)}} + \underbrace{2\dot{f}^{ij}\bar{g}(\bar\nabla_jK,\bar\nabla_i\nu)}_{\text{(B)}} \nonumber\\
& \hskip 0.5cm - \dot{f}^{ij}h_{ij}\kappa  - \underbrace{\dot{f}^{ij}(\nabla_i h_{jl})K_l}_{\text{(C)}} - \dot{f}^{ij} (h^2)_{ij}\bar{g}(K,\nu) \nonumber
\end{align}
Now consider the term (A) above:
\[2\dot{f}^{ij}\kappa\bar{g}([\nu,e_i],e_j) = 2\dot{f}^{ij}\kappa \bar{g}(\bar\nabla_\nu e_i, e_j) - 2\dot{f}^{ij}\kappa\bar{g}(\bar\nabla_i \nu, e_j).\]
Note that $\bar{g}(\bar\nabla_\nu e_i, e_j)$ is anti-symmetric whereas $\dot{f}^{ij}$ is symmetric, we have
\[2\dot{f}^{ij}\kappa \bar{g}(\bar\nabla_\nu e_i, e_j) = 0.\]
Recall that $\bar\nabla_i\nu = -h_{ij}e_j$, we get:
\[- 2\dot{f}^{ij}\kappa\bar{g}(\bar\nabla_i \nu, e_j) = 2\dot{f}^{ij}\kappa h_{ij}.\]
This shows
\begin{equation}
\label{eq:term_A}
2\dot{f}^{ij}\kappa\bar{g}([\nu,e_i],e_j) = 2\kappa\dot{f}^{ij}h_{ij}.	
\end{equation}
For term (B), we let $\lambda_i$'s be the principal curvatures, then we have
\begin{align}
\label{eq:term_B}
& 2\dot{f}^{ij}\bar{g}(\bar\nabla_j K, \bar\nabla_i\nu)\\
& = 2\D{f}{\lambda_i}\delta_{ij}\bar{g}(\bar\nabla_j K, -h_{il}e_l) \nonumber\\
& = -2\D{f}{\lambda_i}\lambda_i\bar{g}(\bar\nabla_i K, e_i) \nonumber\\
& = -\sum_i \D{f}{\lambda_i}\lambda_i \cdot 2\kappa\bar{g}(e_i,e_i) & \text{($K$ is a conformal Killing field)} \nonumber\\
& = -2\kappa\dot{f}^{ij}h_{ij} \nonumber
\end{align}
Finally, we handle the term (C) using Codazzi's equation:
\begin{align}
\overline{\Rm}(e_i,e_l,e_j,\nu) & = \nabla_i h_{jl} - \nabla_l h_{ij} \nonumber\\
\label{eq:term_C}\implies \overline{\Rm}(e_i,K^T,e_j,\nu) & = (\nabla_i h_{jl})K_l - (\nabla_l h_{ij})K_l
\end{align}
where $K^T$ denotes the tangential projection of $K$ onto $T\Sigma$.

Substituting \eqref{eq:term_A}, \eqref{eq:term_B} and \eqref{eq:term_C} back into \eqref{eq:Hess_2}, we get
\begin{align}
\label{eq:Hess_3}
& \dot{f}^{ij}\nabla_i\nabla_j\big(\bar{g}(K,\nu)\big)\\
& =	\dot{f}^{ij}\overline{\Rm}(e_i,\nu,e_j,K-K^T) - \dot{f}^{ij}g_{ij}\bar\nabla_\nu\kappa - \dot{f}^{ij}h_{ij}\kappa  - \dot{f}^{ij}(\nabla_l h_{ij})K_l - \dot{f}^{ij} (h^2)_{ij}\bar{g}(K,\nu)\nonumber\\
& = \dot{f}^{ij} \overline{\Rm}(e_i,\nu,e_j,\bar{g}(K,\nu)\nu) - \dot{f}^{ij}g_{ij}\bar\nabla_\nu\kappa - \dot{f}^{ij}h_{ij}\kappa - (\nabla_l f)K_l - \dot{f}^{ij} (h^2)_{ij}\bar{g}(K,\nu)\nonumber\\
& = \dot{f}^{ij}\overline{\Rm}(e_i,\nu,e_j,\nu)\bar{g}(K,\nu) - \dot{f}^{ij}g_{ij}\bar\nabla_\nu\kappa - \dot{f}^{ij}h_{ij}\kappa - \bar{g}(K,\nabla f) - \dot{f}^{ij} (h^2)_{ij}\bar{g}(K,\nu).\nonumber 
\end{align}
The proposition is proved by combining \eqref{eq:ddt} and \eqref{eq:Hess_3}.
\end{proof}

\begin{corollary}
If $(N^{n+1},\bar{g})$ is a space form with constant sectional curvature $k$, then \eqref{eq:support} becomes
\begin{align}
\label{eq:support_space_form}
& \left(\D{}{t} - \dot{f}^{ij}\nabla_i\nabla_j\right)\bar{g}(K,\nu)\\
& = \dot{f}^{ij}g_{ij}\left(\bar\nabla_\nu\kappa + k\bar{g}(K,\nu)\right) + f\kappa  + \dot{f}^{ij}h_{ij}\kappa + \dot{f}^{ij} (h^2)_{ij}\bar{g}(K,\nu). \nonumber
\end{align}
When $K = X = \phi(\rho)\D{}{\rho}$, then we have
\begin{equation}
\label{eq:support_space_form_X}
\left(\D{}{t} - \dot{f}^{ij}\nabla_i\nabla_j\right)\bar{g}(X,\nu) = f\phi  + \dot{f}^{ij}h_{ij}\phi' + \dot{f}^{ij} (h^2)_{ij}\bar{g}(X,\nu)
\end{equation}
\end{corollary}

\begin{proof}
\eqref{eq:support_space_form} follows directly from the fact that
\[\overline{\Rm}(e_i,\nu,e_j,\nu) = k\big(\bar{g}(e_i,\nu)\bar{g}(\nu,e_j) - \bar{g}(e_i,e_j)\bar{g}(\nu,\nu)\big) = -kg_{ij}.\]
To prove \eqref{eq:support_space_form_X}, we write
$\nu = \nu^\rho\D{}{\rho} + \nu^i\D{}{w_i}$ where $(w_1, \cdots, w_n)$ are local coordinates of $\mathbb{S}^n$, then we have $\kappa = \phi'(\rho)$ and so
\begin{align*}
\bar\nabla_\nu\kappa & = \bar\nabla_{\nu^\rho\D{}{\rho} + \nu^i\D{}{z_i}}\phi'(\rho) = \nu^\rho \phi''(\rho)\\
k\bar{g}(X,\nu) & = k\bar{g}\left(\phi(\rho)\D{}{\rho}, \nu^\rho\D{}{\rho} + \nu^i\D{}{z_i}\right)\\
& = k\nu^\rho \phi(\rho).
\end{align*}
Note that $\phi(\rho)$ satisfies the ODE $\phi''(\rho) + k\phi(\rho) = 0$, and so we have
\[\bar\nabla_\nu\kappa + k\bar{g}(X,\nu) = 0,\]
and so \eqref{eq:support_space_form_X} follows directly from \eqref{eq:support_space_form}.
\end{proof}

\begin{remark}
\eqref{eq:support_space_form_X} can also be proved by combining \eqref{eq:ddt}, Lemma 2.6 in Guan-Li's paper \cite{GL15}, and the fact that $\dot{f}^{ij}(\nabla_k h_{ij})X_k = (\nabla_k f)X_k = \bar{g}(X,\nabla f)$. Note that there is a different sign convention of $h_{ij}$ in \cite{GL15}.	
\end{remark}

\section{Uniqueness of self-similar solutions}
This section we first formally define the notation of self-similar solutions in a warped product space $N^{n+1} = I \times M^n$ with metric
\[\bar{g} = d\rho^2 + \phi(\rho)^2 g_M,\]
where $\phi(\rho)$ is a strictly increasing, positive function of $\rho$. As discussed earlier, the conformal vector field
\[X = \phi(\rho)\D{}{\rho}\]
could play the role of the position vector field in Euclidean space. Denote $\psi_t^X : N \to N$ to be the flow map of $X$ in $N$.

Consider the flow of a hypersurface $\Sigma_t \subset N$ defined as in \eqref{eq:flow-perp}:
\begin{equation*}
\left(\D{F}{t}\right)^\perp = f\nu
\end{equation*}
where $f(\lambda_1, \cdots, \lambda_n) : \Gamma \subset \R^n \to \R$ is a $C^1$ function of principal curvatures satisfying the following conditions:
\label{eq:f_conditions}
\begin{enumerate}[(i)]
\item the domain $\Gamma$ of $f$ is a cone in $\R^n$, i.e. whenever $(\lambda_1, \cdots, \lambda_n) \in \Gamma$ and $c > 0$, we have $(c\lambda_1, \cdots, c\lambda_n) \in \Gamma$.
\item $f$ is a homogeneous function of degree $\deg f$, i.e. for any $(\lambda_1, \cdots, \lambda_n) \in \Gamma$ and $c > 0$, we have
\[f(c\lambda_1, \cdots, c\lambda_n) = c^{\deg f}f(\lambda_1, \cdots, \lambda_n).\]
\item $f$ is strictly increasing along each $\lambda_i$, i.e. $\displaystyle{\D{f}{\lambda_i} > 0}$. This is to ensure the flow is ellitpic.
\end{enumerate}
Notable examples of such functions $f$ include:
\begin{itemize}
\item $f = H = \lambda_1 + \cdots + \lambda_n$: mean curvature flow (MCF)
\item $f = -\frac{1}{H} = -\frac{1}{\lambda_1 + \cdots + \lambda_n}$: inverse mean curvature flow (IMCF)
\item $f = K = \lambda_1 \cdots \lambda_n$: Gauss curvature flow (GCF)
\item $f = \sigma_k$, where $\sigma_k$ is the $k$-th symmetric polynomial of $\lambda_i$'s.	
\end{itemize}
Now we state the definition of self-similar solutions in warped product spaces:
\begin{definition}[Self-similar solutions]
$\Sigma_t$ is called a self-similar solution along $X$ under the flow \eqref{eq:flow-perp} if there exists a time function $\tau(t)$ with either $\tau'(t) > 0$ or $\tau'(t) < 0$ for all $t$, such that under the flow \eqref{eq:flow-perp} starting from $\Sigma_0$, the hypersurface evolves by $\Sigma_t := \psi_{\tau(t)}^X \circ \Sigma_0$.

Furthermore, if $\tau'(t) > 0$, we call such a self-similar solution to be an \emph{expander}; whereas if $\tau'(t) < 0$, we call it a \emph{shrinker}.
\end{definition}
Suppose $F_0$ is a local parametrization of $\Sigma_0$, then $F_t = \psi_{\tau(t)}^X \circ F_0$ is a local parametrization of $\Sigma_t$. By differentiating $F_t$ with respect to $t$, we get:
\[\D{F_t}{t} = \D{}{t}\psi_{\tau(t)}^X \circ F_0 = \tau'(t) (X \circ F_t) \implies \left(\D{F_t}{t}\right)^\perp = \tau'\bar{g}(X \circ F_t, \nu_t)\nu_t.\]
Therefore, a self-similar solution satisfies the equation:
\[f = \tau'\bar{g}(X,\nu).\]
When $\bar{g}(X,\nu) \not= 0$, we say the hypersurface $\Sigma_0$ to be star-shaped (with respect to $X$). According to the above equation, a self-similar solution $\Sigma_0$ would be star-shaped if $f \not= 0$ on $\Sigma_0$.

\begin{theorem}
\label{thm:symmetry}
Suppose $\Sigma_t^n$ is a compact star-shaped self-similar solution in $N^{n+1} = I \times M^n$ along $X$ to the flow \eqref{eq:flow-perp}. Suppose $\Sigma_0$, $\phi$, $\tau$, and $f$ satisfy the condition
\begin{equation}
\label{eq:conditions}
\dot{f}^{ij}g_{ij}\frac{\phi''}{\phi} + \tau'(1+\deg f)\phi' \geq 0,
\end{equation}
then $\Sigma_0$ is tangential to any Killing vector field $K$ on $N^{n+1}$ such that $[X,K] = 0$. Furthermore, if $M^n$ is a compact homogeneous space, then $\Sigma_0$ must be a slice $\{\rho_0\} \times M^n$.
\end{theorem}

\begin{proof}
Consider the flow map $\psi_t^X$ generated by $X$. For simplicity, we denote it by $\psi_t$. Suppose $F_0(u_1,\cdots,u_n)$ is a local parametrization of $\Sigma_0$, then $F_t = \psi_{\tau(t)} \circ F_0$ is the local parametrization of $\Sigma_t$, so that we have
\[\D{F_t}{u_i} = (\psi_{\tau(t)})_*\left(\D{F_0}{u_i}\right).\]
Since the flow map $\psi_{\tau(t)}$ is conformal, $\nu_t$ and $\big(\psi_{\tau(t)}\big)_*\nu_0$ points at the same direction. By renormalization, we must have (up to a sign),
\[\nu_t = \dfrac{\big(\psi_{\tau(t)}\big)_*\nu_0}{\sqrt{\bar{g}\big(\big(\psi_{\tau(t)}\big)_*\nu_0,\big(\psi_{\tau(t)}\big)_*\nu_0\big)}}.\]
It is given that $\mathcal{L}_X K = [X,K] = 0$, so we have for any $t$,
\[\big(\psi_{\tau(t)}\big)_*K = K \circ \psi_{\tau(t)}.\]
This shows
\[\bar{g}(K \circ F_t,\nu_t) = \frac{\bar{g}(K \circ \psi_{\tau(t)} \circ F_0, \big(\psi_{\tau(t)}\big)_*\nu_0)}{\sqrt{\bar{g}\big(\big(\psi_{\tau(t)}\big)_*\nu_0,\big(\psi_{\tau(t)}\big)_*\nu_0\big)}} = \frac{\big(\psi_{\tau(t)}\big)^*\bar{g}(K \circ F_0, \nu_0)}{\sqrt{\big(\psi_{\tau(t)}\big)^*\bar{g}\big(\nu_0,\nu_0\big)}}.\]
Similarly, we also have
\[\bar{g}(X \circ F_t,\nu_t) = \frac{\big(\psi_{\tau(t)}\big)^*\bar{g}(X \circ F_0, \nu_0)}{\sqrt{\big(\psi_{\tau(t)}\big)^*\bar{g}\big(\nu_0,\nu_0\big)}}.\]
Since $\psi_{\tau(t)}^X$ is a conformal map, $(\psi_{\tau(t)}^X)^*\bar{g}$ is given by a $e^{\varphi_t}\bar{g}$ for some scalar function $\varphi_t$. We can conclude that
\[\frac{\bar{g}(K \circ F_t,\nu_t)}{\bar{g}(X \circ F_t,\nu_t)} = \frac{\bar{g}(K \circ F_0, \nu_0)}{\bar{g}(X \circ F_0, \nu_0)},\]
and hence
\[\D{}{t}\frac{\bar{g}(K \circ F_t,\nu_t)}{\bar{g}(X \circ F_t,\nu_t)} = 0.\]
Under a tangential diffeomorphism $\Phi_t : \Sigma_t \to \Sigma_t$, one can convert \eqref{eq:flow-perp} into $\widetilde{F}_t := F_t \circ \Phi_t$ so that $\widetilde{\nu}_t = \nu_t \circ \Phi_t$, and
\[\D{\widetilde{F}}{t} = f\widetilde\nu.\]
Then, we have
\[\frac{\bar{g}(K \circ \widetilde{F}_t, \widetilde\nu_t)}{\bar{g}(X \circ \widetilde{F}_t,\widetilde\nu_t)} = \frac{\bar{g}(K \circ F_t,\nu_t)}{\bar{g}(X \circ F_t,\nu_t)} \circ \Phi_t,\]
and so
\[\D{}{t}\left(\frac{\bar{g}(K \circ \widetilde{F}_t,\widetilde\nu_t)}{\bar{g}(X \circ \widetilde{F}_t,\widetilde\nu_t)}\right) = g\left(\D{\Phi_t}{t}, \nabla\left(\frac{\bar{g}(K \circ F_t,\nu_t)}{\bar{g}(X \circ F_t,\nu_t)}\right)\right)\]
Now abbreviate $\bar{g}(K \circ \widetilde{F}_t, \widetilde\nu_t)$ by simply $\bar{g}(K,\nu)$, and similarly for $\bar{g}(X \circ \widetilde{F}_t,\widetilde\nu_t)$. By \eqref{eq:support}, we have
\begin{align}
\label{eq:box_K}
\left(\D{}{t} - \dot{f}^{ij}\nabla_i\nabla_j\right)\bar{g}(K,\nu) & =  -\dot{f}^{ij}\overline{\Rm}(e_i,\nu,e_j,\nu)\bar{g}(K,\nu) + \dot{f}^{ij} (h^2)_{ij}\bar{g}(K,\nu)\\
\label{eq:box_X}
\left(\D{}{t} - \dot{f}^{ij}\nabla_i\nabla_j\right)\bar{g}(X,\nu) & = \dot{f}^{ij}g_{ij}\bar\nabla_\nu\phi' -\dot{f}^{ij}\overline{\Rm}(e_i,\nu,e_j,\nu)\bar{g}(X,\nu)\\
& \hskip 0.5cm + f\phi'  + \dot{f}^{ij}h_{ij}\phi' + \dot{f}^{ij} (h^2)_{ij}\bar{g}(X,\nu)\nonumber\\
& = \dot{f}^{ij}g_{ij}\phi''d\rho(\nu)-\dot{f}^{ij}\overline{\Rm}(e_i,\nu,e_j,\nu)\bar{g}(X,\nu) + (1 + \deg f)f\phi'\nonumber\\
& \hskip 0.5cm + \dot{f}^{ij} (h^2)_{ij}\bar{g}(X,\nu)\nonumber\\
& = \dot{f}^{ij}g_{ij}\frac{\phi''}{\phi}\bar{g}(X,\nu)-\dot{f}^{ij}\overline{\Rm}(e_i,\nu,e_j,\nu)\bar{g}(X,\nu)\nonumber\\
& \hskip 0.5cm + (1 + \deg f)f\phi' + \dot{f}^{ij} (h^2)_{ij}\bar{g}(X,\nu) \nonumber
\end{align}
Here we have used the fact that for a homogeneous function $f$, we have
\[\dot{f}^{ij}h_{ij} = (\deg f)f.\]
Recall that for a self-similar solution, we have:
\begin{equation}
\label{eq:soliton}
f = \tau'\bar{g}(X,\nu).
\end{equation}
Since it is given that $\tau' \not= 0$ and $f \not= 0$, we can consider the quotient the quotient $\dfrac{\bar{g}(K,\nu)}{\bar{g}(X,\nu)}$. By direct computations, we first have:
\begin{align}
\label{eq:box_quotient}
& \left(\D{}{t} - \dot{f}^{ij}\nabla_i\nabla_j\right)\frac{\bar{g}(K,\nu)}{\bar{g}(X,\nu)}\\
& = \frac{\bar{g}(X,\nu)\left(\D{}{t}-\dot{f}^{ij}\nabla_i\nabla_j\right)(\bar{g}(K,\nu))-\bar{g}(K,\nu)\left(\D{}{t}-\dot{f}^{ij}\nabla_i\nabla_j\right)(\bar{g}(X,\nu))}{\bar{g}(X,\nu)^2} \nonumber\\
& \hskip 0.5cm + 2\dot{f}^{ij}\nabla_i \log \bar{g}(X,\nu) \cdot \nabla_j \left(\frac{\bar{g}(K,\nu)}{\bar{g}(X,\nu)}\right) \nonumber
\end{align}
Consider the numerator of term, by using \eqref{eq:box_K} and \eqref{eq:box_X} we get:
\begin{align*}
& \bar{g}(X,\nu)\left(\D{}{t}-\dot{f}^{ij}\nabla_i\nabla_j\right)(\bar{g}(K,\nu))-\bar{g}(K,\nu)\left(\D{}{t}-\dot{f}^{ij}\nabla_i\nabla_j\right)(\bar{g}(X,\nu))\\
& = \bar{g}(X,\nu)\left(-\dot{f}^{ij}\overline{\Rm}(e_i,\nu,e_j,\nu)\bar{g}(K,\nu) + \dot{f}^{ij} (h^2)_{ij}\bar{g}(K,\nu)\right)\\
& \hskip 0.5cm -\bar{g}(K,\nu)\left(\dot{f}^{ij}g_{ij}\frac{\phi''}{\phi}\bar{g}(X,\nu)-\dot{f}^{ij}\overline{\Rm}(e_i,\nu,e_j,\nu)\bar{g}(X,\nu) + (1 + \deg f)f\phi' + \dot{f}^{ij} (h^2)_{ij}\bar{g}(K,\nu)\right)\\
& = -\bar{g}(K,\nu)\left(\dot{f}^{ij}g_{ij}\frac{\phi''}{\phi}\bar{g}(X,\nu)+(1 + \deg f)f\phi'\right)
\end{align*}
Therefore, putting it back to \eqref{eq:box_quotient}, we get:
\begin{align*}
-\dot{f}^{ij}\nabla_i\nabla_j\left(\frac{\bar{g}(K,\nu)}{\bar{g}(X,\nu)}\right) = -\dot{f}^{ij}g_{ij}\frac{\phi''}{\phi}\frac{\bar{g}(K,\nu)}{\bar{g}(X,\nu)} - \frac{\bar{g}(K,\nu)}{\bar{g}(X,\nu)^2}(1+\deg f)f\phi' + U^i \nabla_i\left(\frac{\bar{g}(K,\nu)}{\bar{g}(X,\nu)}\right)
\end{align*}
for some tangent vector field $U = U^i e_i$. Using \eqref{eq:soliton}, the following holds $\Sigma_0$:
\[-\dot{f}^{ij}\nabla_i\nabla_j\left(\frac{\bar{g}(K,\nu)}{\bar{g}(X,\nu)}\right) = -\left(\dot{f}^{ij}g_{ij}\frac{\phi''}{\phi}+(1+\deg f)\tau'\phi'\right)\frac{\bar{g}(K,\nu)}{\bar{g}(X,\nu)}+ U^i \nabla_i\left(\frac{\bar{g}(K,\nu)}{\bar{g}(X,\nu)}\right).\]
By the maximum principle (note that $\Sigma_0$ is compact), at the point achieving the maximum of $\dfrac{\bar{g}(K,\nu)}{\bar{g}(X,\nu)}$, we have
\[\underbrace{-\dot{f}^{ij}\nabla_i\nabla_j\left(\frac{\bar{g}(K,\nu)}{\bar{g}(X,\nu)}\right)}_{\geq 0} = \underbrace{-\left(\dot{f}^{ij}g_{ij}\frac{\phi''}{\phi}+(1+\deg f)\tau'\phi'\right)}_{\leq 0 \text{ by \eqref{eq:conditions}}}\frac{\bar{g}(K,\nu)}{\bar{g}(X,\nu)} + \underbrace{U^i \nabla_i \left(\frac{\bar{g}(K,\nu)}{\bar{g}(X,\nu)}\right)}_{=0}.\]
This implies $\displaystyle{\max_{\Sigma_0}\frac{\bar{g}(K,\nu)}{\bar{g}(X,\nu)} \leq 0}$. By a similar argument, we also have $\displaystyle{\min_{\Sigma_0}\frac{\bar{g}(K,\nu)}{\bar{g}(X,\nu)} \geq 0}$. This proves $\bar{g}(K,\nu) = 0$ on $\Sigma_0$. Hence, this shows $K$ is a tangent vector field to $\Sigma_0$, completing the proof that $\Sigma_0$ is invariant under the flow map $\psi_t^K$.

Now we further assume $M^n$ is a compact homogeneous space such as the sphere $\mathbb{S}^n$, or compact quotients of the hyperbolic space $\mathbb{H}^n$, then the Lie algebra of Killing vector fields act transitively on $M^n$, and they span the tangent space of $M^n$ at each point.

Take any Killing vector field $K$ in $M^n$. It lifts up to a vector field $\widetilde{K}$ on the product space $I \times M^n$. We argue that $\widetilde{K}$ is also a Killing vector field on $N^{n+1}$, and that $[X, \widetilde{K}] = 0$.

Express $M^n$ in local coordinates $(w_1, \cdots, w_n)$, and write $K$ in local coordinates
\[K = K^i(w_1,\cdots,w_n) \D{}{w_i} \in TM^n,\]
then $(\rho, w_1, \cdots, w_n)$ can be taken to be local coordinates of $N^{n+1} = I \times M^n$, and we can still write the lifting $\widetilde{K}$ as
\[\widetilde{K} = K^i(w_1,\cdots,w_n) \D{}{w_i} \in TN^{n+1}.\]
To prove that $\widetilde{K}$ is still a Killing vector field, we consider
\begin{align*}
\mathcal{L}_{\widetilde{K}}d\rho & = di_{\widetilde{K}}d\rho = d(d\rho(\widetilde{K})) = 0\\
\implies \mathcal{L}_{\widetilde{K}}\bar{g} & = \mathcal{L}_{\widetilde{K}}\big(d\rho^2 + \phi(\rho)^2 g_M\big)\\
& = 0 + \underbrace{\phi(\rho)^2}_{\textup{indep of $w_i$'s}}\mathcal{L}_{\widetilde{K}}g_M\\
& = 0.
\end{align*}
Also, it can also be checked that $[X, \widetilde{K}] = 0$, since
\[[X, \widetilde{K}] = \phi(\rho)\D{K^i}{\rho} - K^i\D{\phi(\rho)}{w_i} = 0.\]
Therefore, any Killing vector field $K$ on $M^n$ lifts up to a Killing vector field $\widetilde{K}$ on $N^{n+1}$ which is tangential to $\Sigma_0$ according to the result we proved. For $M^n$ being a homogeneous space, at every point $(\rho_0, x) \in \Sigma_0 \subset I \times M^n$, the tangent space $T_{(\rho_0,x)}(\{\rho_0\} \times M^n) \cong T_xM^n$ is spanned by (lifted) Killing vector fields on $M^n$, which were already shown to be tangential to $\Sigma_0$. This proves that $T_{(\rho_0,x)}\Sigma_0 = T_{(\rho_0,x)}(\{\rho_0\} \times M^n)$, and so $T_{(\rho_0,x)}\Sigma_0$ is orthogonal to $\D{}{\rho}$. It proves that $\Sigma_0$ is the slice $\{\rho_0\} \times M^n$.

\end{proof}

Here are many examples in which conditions \eqref{eq:conditions} hold. Notably, when $f$ is of homogeneous degree $-1$, the conditions \eqref{eq:conditions} simply require $\phi'' \geq 0$. Note that $\phi > 0$ is always assumed.

For Euclidean spaces $\R^{n+1}$ where $\phi(\rho) = \rho$, it holds true. This recovers the uniqueness result of self-similar solutions of IMCF in Euclidean spaces proved by Drugan-Lee-Wheeler \cite{DLW16}, of the $(\sigma_k/\sigma_j)^{1/(j-k)}$-flow in Euclidean spaces by Kwong-Lee-Pyo \cite{KLP18}, and the general result for any parabolic flow by degree $(-1)$ homogeneous curvature in Euclidean spaces in the joint work \cite{CCF21} by the author. 

The condition $\phi''\geq 0$ also holds true for hyperbolic spaces $\mathbb{H}^{n+1}$ where $\phi(\rho) = \sinh\rho$, so that $\phi''/\phi \equiv 1$. Furthermore, the anti-deSitter-Schwarzschild space of mass $m \geq 0$ is of the form $[s_0,\infty) \times \mathbb{S}^n$ with the metric
\[\frac{1}{\omega(s)}ds^2 + s^2 g_{\mathbb{S}^n}\]
where $\omega(s) = 1 - ms^{1-n} + s^2$. One can verify that it can be rewritten in the warped product form as
\[d\rho^2 + \phi(\rho)^2g_{\mathbb{S}^n}\]
so that $\phi'(\rho) = \sqrt{\omega(s)}$ and $\phi''(\rho) = \frac{1}{2}\omega'(s)$. This shows condition the $\phi'' > 0$ holds true for anti-deSitter-Schwarzschild's space, since
\[\phi''(\rho) = \frac{1}{2}(2s + m(n-1)s^{-n}) > 0.\]
Note that when $m = 0$, the anti-deSitter-Schwarzschild metric becomes the hyperbolic metric. To summarize, we have obtained the following new results:

\begin{corollary}
\label{cor:uniqueness}
When $N^{n+1} = I \times M^n$ is a warped product by a compact homogeneous space $(M^n, g_M)$ with the metric $\bar{g} = d\rho^2 + \phi(\rho)^2 g_M$ such that $\phi''(\rho) \geq 0$, the only compact star-shaped self-similar solution in $N$ to the flow \eqref{eq:flow-perp}, with the speed function $f$ satisfying conditions (i)-(iii) in p.\pageref{eq:f_conditions} with $\deg f = -1$, must be a slice $\{\rho_0\} \times M^n$. In particular, this is true for anti-deSitter-Schwarzschild spaces of mass $m \geq 0$, including the hyperbolic space $\mathbb{H}^{n+1}$ when $m = 0$. 
\end{corollary}

There are some other spaces in general relativity which are of the warped product form, with $\omega(s) = 1 - ms^{1-n} - \kappa s^2$ for deSitter-Schwarzschild space, and $\omega(s) = 1 - ms^{1-n} + q^2 s^{2-2n}$ for the Reissner-Nordstrom's space. The condition $\phi''(\rho) = \frac{1}{2}\omega'(s) > 0$ holds only for part of the region when $s$ is sufficiently small, hence a similar uniqueness result as Corollary \eqref{cor:uniqueness} would hold only $\Sigma_0$ is in this region.

For curvature flows with $\deg f > -1$, \eqref{eq:conditions} holds when $\tau' > 0$, $\phi' > 0$, and $\phi'' > 0$. These are still true $\Sigma_0$ is an expanding self-similar solution in hyperbolic space and anti-deSitter-Schwarzschild spaces. Let's state it as another corollary:

\newpage
\begin{corollary}
\label{cor:uniqueness_deg>-1}
When $N^{n+1} = I \times M^n$ is a warped product by a compact homogeneous space $(M^n, g_M)$ with the metric $\bar{g} = d\rho^2 + \phi(\rho)^2 g_M$ such that $\phi'(\rho) \geq 0$ and $\phi''(\rho) \geq 0$, the only compact star-shaped self-similar solution in $N$ to the flow \eqref{eq:flow-perp}, with the speed function $f$ satisfying conditions (i)-(iii) in p.\pageref{eq:f_conditions} with $\deg f > -1$, must be a slice $\{\rho_0\} \times M^n$. In particular, this is true for anti-deSitter-Schwarzschild spaces of mass $m \geq 0$, including the hyperbolic space $\mathbb{H}^{n+1}$ when $m = 0$. 
\end{corollary}

\section{Symmetry of non-compact self-similar solutions in doubly-warped spaces}
In this section, we discuss some rigidity results of non-compact self-similar solutions. The application of the maximum principle is not as straight-forward as in the compact case. Yet, using some similar idea as in \cite{FM19}, one can still prove some rotational symmetry results for self-similar solutions asymptotic to a ``cone'' in some doubly-warped product spaces.

Now further assume the factor $(M^n,g_M)$ in $N^{n+1} = I \times M^n$ is itself a warped product space, i.e. $M^n = J \times S^{n-1}$ where $J$ is an interval, and $(S^{n-1}, g_S)$ is an $(n-1)$-dimensional Riemannian manifold, so that $g_M$ takes the form:
\[g_M = d\theta^2 + r(\theta)^2g_S,\]
where $\theta \in J$, and $r(\theta)$ is a strictly increasing function on $J$.

As such, $N^{n+1}$ can be regarded as a doubly-warped product space, in a sense that $N^{n+1} = I \times J \times S^{n-1}$, and its metric is given by
\[\bar{g} = d\rho^2 + \phi(\rho)^2 \, d\theta^2 + \phi(\rho)^2 r(\theta)^2 g_S.\]
The subset $\mathcal{C}(\theta_0)$, with $\theta_0$ fixed, defined by
\[\mathcal{C}(\theta_0) = \{(\rho, \theta_0, w) \in N : \rho \in I, w \in S\}\]
is called a ``cone'' with angle $\theta_0$. By straight-forward computation, the principal curvatures of $\mathcal{C}(\theta_0)$ are given by
\[\left\{0, \frac{1}{\phi(\rho)}\frac{r'(\theta_0)}{r(\theta_0)}, \cdots, \frac{1}{\phi(\rho)}\frac{r'(\theta_0)}{r(\theta_0)}\right\}.\]
If a complete non-compact hypersurface $\Sigma$ has an end that is sufficiently close, say in $C^2$-sense, to $\mathcal{C}(\theta_0)$ when $\rho$ is large, then it is expected that the principal curvatures of $\Sigma$ would be close to those of $\mathcal{C}(\theta_0)$, and that $\Sigma$ should be approximately as symmetric as $\mathcal{C}(\theta_0)$. Using this idea, we could then prove the some symmetry results about non-compact self-similar solutions. We first derive a proposition in a general setting, and will apply it to prove some symmetry results for some specific warped product spaces and curvature flows.

\begin{proposition}
\label{prop:non-compact}
Let $N^{n+1} = [0,\infty) \times J \times S^{n-1}$ be a doubly-warped product space, with metric of the form
\[\bar{g} = d\rho^2 + \phi(\rho)^2\,d\theta^2 + \phi(\rho)^2 r(\theta)^2 g_S\]
where $\rho \in [0,\infty)$, $\theta \in J$, and $g_S$ is a Riemannian metric on $g_S$. Assume for each $\rho_0 > 0$
\[N \cap \{\rho \leq \rho_0\} = \{(\rho, \theta, w) : \rho \leq \rho_0, \theta \in J, w \in S\}\]
is compact. Suppose $\Sigma_t^n$ is a complete non-compact star-sharped self-similar solution on $(N^{n+1},\bar{g})$ evolving along the flow \eqref{eq:flow-perp}
\begin{equation*}
\left(\D{F}{t}\right)^\perp = f\nu
\end{equation*}
where $f(\lambda_1,\cdots,\lambda_n)$ is a $C^1$ function satisfying (i)-(iii) in p.\pageref{eq:f_conditions}. Assume there exists $\varepsilon > 0$ such that $\Sigma_0$, $\phi$, $\tau$, and $f$ satisfies the condition
\begin{equation}
\label{eq:condition_epsilon}
\left(f\phi'(1+\deg f) + \dot{f}^{ij}g_{ij}\frac{\phi''f}{\phi\tau'}\right) > \varepsilon\Big(\dot{f}^{ij}(h^2)_{ij} + \dot{f}^{ij}\overline{\Rm}(e_i,\nu,\nu,e_j) - \tau'\phi'\Big).
\end{equation}
Then, if a Killing vector field $K$ on $N^{n+1}$ satisfies $[X,K] = 0$ and
\begin{equation}
\label{eq:asymptotics}
\lim_{\rho_0 \to +\infty} \sup_{\Sigma_0 \cap \{\rho > \rho_0\}} \abs{\bar{g}(K,\nu)} = 0,
\end{equation}
then $\Sigma_0$ is tangential to $K$.
\end{proposition}

\begin{proof}
Recall that the conformal Killing vector field $\displaystyle{X = \phi(\rho)\D{}{\rho}}$ induces a flow map $\psi_t$ such that
\[\D{}{t}\psi_t^*\bar{g} = \psi_t^*(\mathcal{L}_X\bar{g}) = (\psi_t)^*(2\phi'(\rho)\bar{g}) = 2(\phi' \circ \rho \circ \psi_t)\psi_t^*\bar{g},\]
so $\psi_t^*\bar{g}$ is conformally related to $\bar{g}$ by the relation:
\[\psi_t^*\bar{g} = \exp\left(2\int_0^t \phi' \circ \rho \circ \psi_s\,ds\right)\bar{g}.\]
Let $F_t$ be the parametrization of the self-similar solution $\Sigma_t$, i.e. $F_t = \psi_{\tau(t)} \circ F_0$ for some function $\tau(t)$ of $t$. Given any conformal Killing field $K$ such that $[X,K] = 0$, then as discussed in the proof of Theorem \ref{thm:symmetry}, we have
\begin{align*}
\bar{g}(K \circ F_t, \nu_t) & = \frac{(\psi_\tau^*\bar{g})(K \circ F_0, \nu_0)}{\sqrt{(\psi_\tau^*\bar{g})(\nu_0,\nu_0)}}\\
& = \exp\left(\int_0^{\tau(t)} \phi' \circ \rho \circ \psi_s\,ds\right)\bar{g}(K \circ F_0,\nu_0).
\end{align*}
From here, we have
\[\D{}{t}\bar{g}(K \circ F_t, \nu_t) = (\tau'\phi' \circ \rho \circ \psi_t)\bar{g}(K \circ F_0,\nu_0).\]
By a suitable tangential diffeomorphism $\Phi_t$, one can arrange that $\widetilde{F}_t := F_t \circ \Phi_t$ satisfies the flow
\[\D{\widetilde{F}}{t} = f\widetilde{\nu}.\]
Precisely, $\Phi_t$ satisfies the ODE:
\begin{align*}
\D{\Phi_t}{t} & = -\left(\D{F_t}{t}\right)^T \circ \Phi_t\\
\Phi_0 & = \id
\end{align*}
where $T$ denotes the tangential projection onto $\Sigma_t$. As $F_t$ evolves by $\psi_\tau \circ F_0$, we also have
\[\D{\Phi_t}{t} = -\tau' X^T \circ F_t\]
By the chain rule, we then have
\begin{align*}
\D{}{t}\bar{g}(K \circ \widetilde{F}_t, \widetilde{\nu}_t) & = \D{}{t} \bar{g}(K \circ F_t, \nu_t) \circ \Phi_t\\
& = \left(\D{}{t}\bar{g}(K \circ F_t, \nu_t)\right) \circ \Phi_t + \bar{g}\left(\nabla\big(\bar{g}(K\circ\widetilde{F}_t, \widetilde{\nu}_t)\big), \D{\Phi_t}{t}\right)\\
& = (\tau'\phi' \circ \rho \circ \psi_t \circ \Phi_t)\bar{g}(K \circ F_0, \nu_0) \circ \Phi_t - \bar{g}\left(\nabla\big(\bar{g}(K\circ\widetilde{F}_t, \widetilde{\nu}_t)\big), \tau' X^T \circ F_t\right).
\end{align*}
Evaluate at $t = 0$, we get
\[\D{}{t}\bigg|_{t=0}\bar{g}(K \circ F_t, \nu_t) = \tau'\phi'\bar{g}(K,\nu) - \tau'\nabla_{X^T}\bar{g}(K,\nu).\]
Hence, \eqref{eq:box_K} becomes
\begin{align*}
& \tau'\phi'\bar{g}(K,\nu) - \tau'\nabla_{X^T}\big(\bar{g}(K,\nu)\big) - \dot{f}^{ij}\nabla_i\nabla_j\bar{g}(K,\nu)\\
& =  -\dot{f}^{ij}\overline{\Rm}(e_i,\nu,e_j,\nu)\bar{g}(K,\nu) + \dot{f}^{ij} (h^2)_{ij}\bar{g}(K,\nu)
\end{align*}
and \eqref{eq:box_X} becomes
\begin{align*}
& \tau'\phi'\bar{g}(X,\nu) - \tau'\nabla_{X^T}\big(\bar{g}(K,\nu)\big) - \dot{f}^{ij}\nabla_i\nabla_j\bar{g}(X,\nu)\\
& = \dot{f}^{ij}g_{ij}\frac{\phi''}{\phi}\bar{g}(X,\nu)-\dot{f}^{ij}\overline{\Rm}(e_i,\nu,e_j,\nu)\bar{g}(X,\nu)\nonumber\\
& \hskip 0.5cm + (1 + \deg f)f\phi' + \dot{f}^{ij} (h^2)_{ij}\bar{g}(X,\nu)\\
& = \dot{f}^{ij}g_{ij}\frac{\phi''}{\phi}\bar{g}(X,\nu)-\dot{f}^{ij}\overline{\Rm}(e_i,\nu,e_j,\nu)\bar{g}(X,\nu)\nonumber\\
& \hskip 0.5cm + \tau'(1 + \deg f)\phi'\bar{g}(X,\nu) + \dot{f}^{ij} (h^2)_{ij}\bar{g}(X,\nu)
\end{align*}
These results motivate us to define the second-order elliptic operator:
\[L:= \dot{f}^{ij}\nabla_i\nabla_j + \tau'\nabla_{X^T}+ \dot{f}^{ij}(h^2)_{ij} - \dot{f}^{ij}\overline{\Rm}(e_i,\nu,e_j,\nu).\]
Then, \eqref{eq:box_K} and \eqref{eq:box_X} can be written as
\begin{align}
\label{eq:L_K} L\big(\bar{g}(K,\nu)\big) & = \tau'\phi'\bar{g}(K,\nu)\\
\label{eq:L_X} L\big(\bar{g}(X,\nu)\big) & = -\dot{f}^{ij}g_{ij}\frac{\phi''}{\phi}\bar{g}(X,\nu) - \tau'\phi'(\deg f)\bar{g}(X,\nu).
\end{align}
By straight-forward computation, for any scalar functions $\varphi_1$, and $\varphi_2 \not= 0$, we have
\begin{align}
\label{eq:Hessian_term_noncompact}
& \dot{f}^{ij}\nabla_i\nabla_j	\left(\frac{\varphi_1}{\varphi_2}\right)\\
& = \frac{\varphi_2 \dot{f}^{ij}\nabla_i\nabla_j\varphi_1 - \varphi_1\dot{f}^{ij}\nabla_i\nabla_j\varphi_2}{\varphi_2^2} - \frac{2}{\varphi_2}\dot{f}^{ij}\nabla_i\left(\frac{\varphi_1}{\varphi_2}\right)\nabla_j\varphi_2 \nonumber\\
& = \frac{1}{\varphi_2^2}\left\{\varphi_2\big(L\varphi_1 -\tau'\nabla_{X^T}\varphi_1 - \dot{f}^{ij}(h^2)_{ij}\varphi_1 + \dot{f}^{ij}\overline{\Rm}(e_i,\nu,e_i,\nu)\varphi_1\big)\right. \nonumber\\
& \hskip 1.5cm \left. -\varphi_1\big(L\varphi_2 -\tau'\nabla_{X^T}\varphi_2 - \dot{f}^{ij}(h^2)_{ij}\varphi_2 + \dot{f}^{ij}\overline{\Rm}(e_i,\nu,e_j,\nu)\varphi_2\big)\right\} \nonumber\\
& \hskip 0.5cm - \frac{2}{\varphi_2}\dot{f}^{ij}\nabla_i\left(\frac{\varphi_1}{\varphi_2}\right)\nabla_j\varphi_2 \nonumber\\
& = \frac{\varphi_2 L\varphi_1 - \varphi_1 L\varphi_2}{\varphi_2^2} - \tau'\nabla_{X^T}\left(\frac{\varphi_1}{\varphi_2}\right) - \frac{2}{\varphi_2^2}\dot{f}^{ij}\nabla_i\left(\frac{\varphi_1}{\varphi_2}\right)\nabla_j\varphi_2 \nonumber
\end{align}
Next we apply the above formula with $\varphi_1 = \bar{g}(K,\nu)$ and $\varphi_2 = \bar{g}(X,\nu) + \varepsilon$. Note that $\Sigma_0$ is star-shaped, so $\varphi_2 > \varepsilon > 0$. Consider the term
\begin{align}
\label{eq:numerator}
& \varphi_2 L\varphi_1 - \varphi_1 L\varphi_2\\
& = \big(\bar{g}(X,\nu)+\varepsilon\big)L\big(\bar{g}(K,\nu)\big)-\bar{g}(K,\nu)L\big(\bar{g}(X,\nu)+\varepsilon\big) \nonumber\\
& = \Big(\bar{g}(X,\nu)L\big(\bar{g}(K,\nu)\big)-\bar{g}(K,\nu)L\big(\bar{g}(K,\nu)\big)\Big) - \varepsilon\Big(\bar{g}(K,\nu)L(1) - L\big(\bar{g}(K,\nu)\big)\Big) \nonumber\\
& = \left(\tau'\phi'(1+\deg f) + \dot{f}^{ij}g_{ij}\frac{\phi''}{\phi}\right)\bar{g}(K,\nu)\bar{g}(X,\nu) \nonumber\\
& \hskip 1.5cm -\varepsilon\Big(\dot{f}^{ij}(h^2)_{ij} + \dot{f}^{ij}\overline{\Rm}(e_i,\nu,\nu,e_j) - \tau'\phi'\Big)\bar{g}(K,\nu) \nonumber\\
& = \underbrace{\left\{\left(f\phi'(1+\deg f) + \dot{f}^{ij}g_{ij}\frac{\phi''f}{\phi\tau'}\right)-\varepsilon\Big(\dot{f}^{ij}(h^2)_{ij} + \dot{f}^{ij}\overline{\Rm}(e_i,\nu,\nu,e_j) - \tau'\phi'\Big)\right\}}_{=: P(\Sigma_0, \phi, f, \tau, \varepsilon)}\bar{g}(K,\nu). \nonumber
\end{align}
By the given condition \eqref{eq:condition_epsilon}, the term $P$ defined above is positive. Now we consider the quotient
\[u := \frac{\varphi_1}{\varphi_2} = \frac{\bar{g}(K, \nu)}{\bar{g}(X, \nu) + \varepsilon}.\]
By the star-shaped property, we have $\bar{g}(X,\nu) + \varepsilon > \varepsilon > 0$. Combine with the asymptotic condition \eqref{eq:asymptotics} on $\bar{g}(K,\nu)$, we have
\begin{equation}
\label{eq:decay}
\lim_{\rho_0 \to +\infty}\sup_{\Sigma_0 \cap \{\rho > \rho_0\}}\abs{u} = 0.	
\end{equation}
We want to argue that $u \equiv 0$ on $\Sigma_0$. Suppose not, there exists $p \in \Sigma_0$ such that $u(p) \not= 0$. Without loss of generality, we assume $u(p) > 0$ (the case $u(p) < 0$ is similar). From \eqref{eq:decay}, there exists a sufficiently large $\rho_0 > 0$ such that
\[u(q) < \frac{1}{2}u(p), \;\; \forall q \in \Sigma_0 \cap \{\rho > \rho_0\}.\]
In particular, $p \in \Sigma_0 \cap \{\rho \leq \rho_0\}$. Since is $\Sigma_0 \cap \{\rho \leq \rho_0\}$ is compact, the above condition implies that $u$ must have an interior maximum. However, \eqref{eq:Hessian_term_noncompact} and \eqref{eq:numerator} would imply at this local maximum point $p' \in \Sigma_0 \cap \{\rho \leq \rho_0\}$, we have
\begin{align*}
0 & \geq \dot{f}^{ij}\nabla_i\nabla_j u = \underbrace{P(\Sigma_0,\phi,f,\tau)}_{>0}\underbrace{\frac{u(p')}{\varphi_2}}_{>0}  \underbrace{- \tau'\nabla_{X^T}u - \frac{2}{\varphi_2^2}\dot{f}^{ij}(\nabla_i u)(\nabla_j\varphi_2)}_{=0} > 0,
\end{align*}
leading to a contradiction. This proves $u \leq 0$. The proof of $u \geq 0$ is similar. It concludes that $\bar{g}(K,\nu) = 0$, hence $K$ is tangential to $\Sigma_0$.
\end{proof}

There are some good examples that the condition \eqref{eq:condition_epsilon} holds. For instance, when $N$ is the Euclidean space $\R^{n+1}$, i.e. $\phi(\rho) = \rho$, $f = H$, $\tau' > 0$, and that $\Sigma_0$ is asymptotic to a cone. A star-shaped self-similar solution $\Sigma_0$ under this setting is a MCF self-expander with positive mean curvature. In this case, we have
\[P = 2H - \varepsilon(\abs{A}^2 - \tau')\]
The asymptotically conical assumption implies $\abs{A} \to 0$ at infinity, we have $\abs{A}^2 < \tau'$ on $\Sigma_0 \cap \{\rho > \rho_0\}$ for some $\rho_0$, so that $P > 0$ on $\Sigma_0 \cap \{\rho > \rho_0\}$ for any $\varepsilon > 0$. Then, by the compactness of $\Sigma_0 \cap \{\rho \leq \rho_0\}$ and the positivity of $H$, one can choose
\[\varepsilon = \frac{\inf_{\Sigma_0 \cap \{\rho \leq \rho_0\}}H}{\sup_{\Sigma_0 \cap \{\rho \leq \rho_0\}}(\abs{A}^2-\tau')} > 0,\]
then we have $P > 0$ on $\Sigma_0 \cap \{\rho \leq \rho_0\}$ as well. Then, applying the proposition with an arbitrary Killing field $K$ of $\mathbb{S}^{n-1}$ in $\R^{n+1} = [0,\infty) \times [0,\pi] \times \mathbb{S}^{n-1}$ (cylindrical coordinates), one can prove that $\Sigma_0$ is rotationally symmetric. It recovers the result proved in \cite{FM19}.

Next we apply Proposition \ref{prop:non-compact} to prove some new symmetry results about expanding self-similar solutions in hyperbolic and anti-deSitter-Schwarzschilds space. These spaces can be expressed as a doubly-warped product as
\[N^{n+1} = [s_0,\infty) \times [0,\pi] \times \mathbb{S}^{n-1}\]
with the metric
\[\frac{1}{1 - ms^{1-n} + s^2}ds^2 + s^2 (d\theta^2 + (\sin^2\theta)g_{\mathbb{S}^{n-1}})\]
for some $m \geq 0$. It can be rewritten in the warped product form as $[0,\infty) \times [0,\pi] \times \mathbb{S}^{n-1}$ with metric
\[d\rho^2 + \phi(\rho)^2(d\theta^2 + (\sin^2\theta)g_{\mathbb{S}^{n-1}})\]
so that $\phi(\rho) = s$, $\phi'(\rho) = \sqrt{1 - ms^{1-n} + s^2}$, and $\phi''(\rho) = \frac{m(n-1)}{2}s^{-n}+s$. Note that when $m = 0$, the space is the hyperbolic space $\mathbb{H}^{n+1}$.

As an application of Proposition \ref{prop:non-compact}, we prove the following result for flows by positive powers of mean curvature. We expect similar results should hold, with some slight modifications, to other flows by homogeneous speed functions of principal curvatures with positive degrees.
\begin{theorem}
\label{thm:AC}
Suppose $\Sigma_t^n$ is a complete non-compact star-sharped expanding self-similar solution to the $\alpha$-mean curvature flow $\D{F}{t} = H^\alpha\nu$, where $\alpha > 0$, on the anti-deSitter-Schwarzschild's space (including hyperbolic space) which is $C^2$-asymptotic to a cone $\mathcal{C}(\theta_0)$, then $\Sigma_0$ must be rotationally symmetric.
\end{theorem}

\begin{proof}
Consider that $f = H^\alpha$, where $\alpha > 0$, we have
\begin{align*}
\dot{f}^{ij} & = \alpha H^{\alpha-1}g^{ij},\\
\dot{f}^{ij}(h^2)_{ij} & = \alpha H^{\alpha-1}\abs{A}^2.
\end{align*}
Also, on the anti-deSitter-Schwarzschild space, the Riemann curvature tensor of $\bar{g}$ satisfies the asymptotics:
\[\overline{\Rm}(e_i, \nu, \nu, e_j) = -\left(\bar{g}(e_i,e_j)\bar{g}(\nu,\nu) - \bar{g}(e_i,\nu)\bar{g}(e_j,\nu)\right) + o(1) = -g_{ij} + o(1)\]
as $\rho \to +\infty$. Combining these, one can check that
\begin{align*}
& P(\Sigma_0, \phi, f, \tau, \varepsilon)\\
& = \left(f\phi'(1+\deg f) + \dot{f}^{ij}g_{ij}\frac{\phi''f}{\phi\tau'}\right)-\varepsilon\Big(\dot{f}^{ij}(h^2)_{ij} + \dot{f}^{ij}\overline{\Rm}(e_i,\nu,\nu,e_j) - \tau'\phi'\Big)\\
& = \left((1+\alpha)H^\alpha \phi' + n\alpha H^{2\alpha-1}\frac{\phi''}{\phi'\tau'}\right)-\varepsilon\left(\alpha H^{\alpha-1}\abs{A}^2 + \alpha H^{\alpha-1}(-n + o(1)) - \tau'\phi'\right)
\end{align*}
If $\Sigma_0$ is asymptotic to a cone $\mathcal{C}(\theta_0) = \{(\rho, \theta_0, w) : \rho \in [0,\infty), w \in \mathbb{S}^{n-1}\}$ in $C^2$-sense, then the principal curvatures of $\Sigma_0$ are asymptotic to those of $\mathcal{C}(\theta_0)$, which are
\[\left\{0, \frac{\cot\theta_0}{\phi(\rho)}, \cdots, \frac{\cot\theta_0}{\phi(\rho)}\right\}.\]
As $\phi(\rho) \to \infty$ as $\rho \to \infty$ for any anti-deSitter-Schwarzschild's space, one has
\[\lim_{\rho_0 \to \infty} \sup_{\Sigma_0 \cap \{\rho > \rho_0\}}\abs{A}^2 = 0.\]
Therefore, there exists $\rho_0 > 0$ such that
\[\abs{A}^2 < \frac{n}{2} \qquad \text{on $\Sigma_0 \cap \{\rho > \rho_0\}$}.\]
For expanding self-similar solution $\Sigma_0$ (i.e $\tau' > 0$) which is star-shaped, then we have $H^\alpha = \tau'\bar{g}(X,\nu) > 0$. Since $\phi(\rho), \phi'(\rho), \phi''(\rho) > 0$ for anti-deSitter-Schwarzschild space, one has
\begin{align*}
(1+\alpha)H^\alpha \phi' + n\alpha H^{2\alpha-1}\frac{\phi''}{\phi'\tau'} & > 0\\
\alpha H^{\alpha-1}\abs{A}^2 + \alpha H^{\alpha-1}(-n + o(1)) - \tau'\phi' & < 0
\end{align*}
on $\Sigma_0 \cap \{\rho > \rho_0\}$. This shows $P(\Sigma_0, \phi, f, \tau, \varepsilon) > 0$ on $\Sigma_0 \cap \{\rho > \rho_0\}$ for any $\varepsilon > 0$. Since $\Sigma_0 \cap \{\rho \leq \rho_0\}$ is compact, one can choose a sufficiently small $\varepsilon > 0$ such that $P(\Sigma_0, \phi, f, \tau, \varepsilon) > 0$ on $\Sigma_0 \cap \{\rho \leq \rho_0\}$. Therefore, the condition \eqref{eq:condition_epsilon} holds, and we can applying Proposition \ref{prop:non-compact}.

Now consider any Killing field $K$ of $\mathbb{S}^{n-1}$. By the same argument at the end of the proof of Theorem \ref{thm:symmetry}, $K$ lifts to a Killing field $\widetilde{K}$ of $N^{n+1} = [0, \infty) \times [0,\pi] \times \mathbb{S}^{n-1}$, and that $[X,\widetilde{K}] = 0$. By the asymptotically concial condition on $\Sigma_0$, condition \eqref{eq:decay} as holds. Therefore, Proposition \ref{prop:non-compact} then shows $\widetilde{K}$ is tangential to $\Sigma_0$. Note that the Killing field acts transitively on $\mathbb{S}^{n-1}$, hence for any $(\bar\rho,\bar\theta,\bar{w}) \in \Sigma_0$, and any $w' \not= \bar{w}$ in $\mathbb{S}^{n-1}$, there exists a Killing field $K$ of $\mathbb{S}^{n-1}$ such that $\psi_t^K(\bar{w}) = w'$ for some $t$. Since $K$ lifts up to a Killing field $\widetilde{K}$ in $N^{n+1}$ which is tangential to $\Sigma_0$, one has $w' = \psi_t^{\widetilde{K}}(\bar{w}) \in \Sigma_0$. In other words, whenever $\Sigma_0$ contains $(\bar\rho, \bar\theta, \bar{w})$, it must contain $\{(\bar\rho, \bar\theta)\} \times \mathbb{S}^{n-1}$. Hence, $\Sigma_0$ is rotationally symmetric.
\end{proof}

\textbf{Acknowledgement:} The author and this research project is partially supported by the General Research Fund \#16304220 by the Hong Kong Research Grants Council.

\bibliographystyle{amsplain}
\bibliography{citations.bib}

\providecommand{\bysame}{\leavevmode\hbox to3em{\hrulefill}\thinspace}
\providecommand{\MR}{\relax\ifhmode\unskip\space\fi MR }
% \MRhref is called by the amsart/book/proc definition of \MR.
\providecommand{\MRhref}[2]{%
  \href{http://www.ams.org/mathscinet-getitem?mr=#1}{#2}
}
\providecommand{\href}[2]{#2}
\begin{thebibliography}{10}

\bibitem{BN20}
Valerii Berestovskii and Yurii Nikonorov, \emph{Riemannian manifolds and
  homogeneous geodesics}, Springer Monographs in Mathematics, Springer, Cham,
  [2020] \copyright 2020. \MR{4179589}

\bibitem{B13}
Simon Brendle, \emph{Constant mean curvature surfaces in warped product
  manifolds}, Publ. Math. Inst. Hautes \'{E}tudes Sci. \textbf{117} (2013),
  247--269. \MR{3090261}

\bibitem{B16}
\bysame, \emph{Embedded self-similar shrinkers of genus 0}, Ann. of Math. (2)
  \textbf{183} (2016), no.~2, 715--728. \MR{3450486}

\bibitem{BCD17}
Simon Brendle, Kyeongsu Choi, and Panagiota Daskalopoulos, \emph{Asymptotic
  behavior of flows by powers of the {G}aussian curvature}, Acta Math.
  \textbf{219} (2017), no.~1, 1--16. \MR{3765656}

\bibitem{BE13}
Simon Brendle and Michael Eichmair, \emph{Isoperimetric and {W}eingarten
  surfaces in the {S}chwarzschild manifold}, J. Differential Geom. \textbf{94}
  (2013), no.~3, 387--407. \MR{3080487}

\bibitem{BHW16}
Simon Brendle, Pei-Ken Hung, and Mu-Tao Wang, \emph{A {M}inkowski inequality
  for hypersurfaces in the anti--de {S}itter--{S}chwarzschild manifold}, Comm.
  Pure Appl. Math. \textbf{69} (2016), no.~1, 124--144. \MR{3433631}

\bibitem{CCF21}
Tsz-Kiu~Aaron Chow, Ka-Wing Chow, and Frederick Tsz-Ho Fong,
  \emph{Self-expanders to inverse curvature flows by homogeneous functions},
  Comm. Anal. Geom. \textbf{29} (2021), no.~2, 329--362. \MR{4250325}

\bibitem{CM12}
Tobias~H. Colding and William~P. Minicozzi, II, \emph{Generic mean curvature
  flow {I}: generic singularities}, Ann. of Math. (2) \textbf{175} (2012),
  no.~2, 755--833. \MR{2993752}

\bibitem{DLW16}
Gregory Drugan, Hojoo Lee, and Glen Wheeler, \emph{Solitons for the inverse
  mean curvature flow}, Pacific J. Math. \textbf{284} (2016), no.~2, 309--326.
  \MR{3544302}

\bibitem{FR23}
Joshua Flynn and Jacob Reznikov, \emph{General conformally induced mean
  curvature flow}, preprint (2023).

\bibitem{F20}
Frederick Tsz-Ho Fong, \emph{Uniqueness and rigidity of self-expanders of
  curvature flows}, Proceedings of the {I}nternational {C}onsortium of
  {C}hinese {M}athematicians 2017, Int. Press, Boston, MA, [2020] \copyright
  2020, pp.~283--300. \MR{4251115}

\bibitem{FM19}
Frederick Tsz-Ho Fong and Peter McGrath, \emph{Rotational symmetry of
  asymptotically conical mean curvature flow self-expanders}, Comm. Anal. Geom.
  \textbf{27} (2019), no.~3, 599--618. \MR{4003004}

\bibitem{G23}
Shanze Gao, \emph{Closed self-similar solutions to flows by negative powers of
  curvature}, J. Geom. Anal. \textbf{33} (2023), no.~12, Paper No. 370, 16.
  \MR{4646418}

\bibitem{GLM18}
Shanze Gao, Haizhong Li, and Hui Ma, \emph{Uniqueness of closed self-similar
  solutions to {$\sigma_k^{\alpha}$}-curvature flow}, NoDEA Nonlinear
  Differential Equations Appl. \textbf{25} (2018), no.~5, Paper No. 45, 26.
  \MR{3845754}

\bibitem{GLW22}
Shanze Gao, Haizhong Li, and Xianfeng Wang, \emph{Self-similar solutions to
  fully nonlinear curvature flows by high powers of curvature}, J. Reine Angew.
  Math. \textbf{783} (2022), 135--157. \MR{4373243}

\bibitem{GM19}
Shanze Gao and Hui Ma, \emph{Self-similar solutions of curvature flows in
  warped products}, Differential Geom. Appl. \textbf{62} (2019), 234--252.
  \MR{3892111}

\bibitem{GM21}
\bysame, \emph{Characterizations of umbilic hypersurfaces in warped product
  manifolds}, Front. Math. China \textbf{16} (2021), no.~3, 689--703.
  \MR{4277378}

\bibitem{G90}
Claus Gerhardt, \emph{Flow of nonconvex hypersurfaces into spheres}, J.
  Differential Geom. \textbf{32} (1990), no.~1, 299--314. \MR{1064876}

\bibitem{GL15}
Pengfei Guan and Junfang Li, \emph{A mean curvature type flow in space forms},
  Int. Math. Res. Not. IMRN (2015), no.~13, 4716--4740. \MR{3439091}

\bibitem{GLW19}
Pengfei Guan, Junfang Li, and Mu-Tao Wang, \emph{A volume preserving flow and
  the isoperimetric problem in warped product spaces}, Trans. Amer. Math. Soc.
  \textbf{372} (2019), no.~4, 2777--2798. \MR{3988593}

\bibitem{H84}
Gerhard Huisken, \emph{Flow by mean curvature of convex surfaces into spheres},
  J. Differential Geom. \textbf{20} (1984), no.~1, 237--266. \MR{772132}

\bibitem{H90}
\bysame, \emph{Asymptotic behavior for singularities of the mean curvature
  flow}, J. Differential Geom. \textbf{31} (1990), no.~1, 285--299.
  \MR{1030675}

\bibitem{HI01}
Gerhard Huisken and Tom Ilmanen, \emph{The inverse mean curvature flow and the
  {R}iemannian {P}enrose inequality}, J. Differential Geom. \textbf{59} (2001),
  no.~3, 353--437. \MR{1916951}

\bibitem{KLP18}
Kwok-Kun Kwong, Hojoo Lee, and Juncheol Pyo, \emph{Weighted
  {H}siung-{M}inkowski formulas and rigidity of umbilical hypersurfaces}, Math.
  Res. Lett. \textbf{25} (2018), no.~2, 597--616. \MR{3826837}

\bibitem{McC11}
James~Alexander McCoy, \emph{Self-similar solutions of fully nonlinear
  curvature flows}, Ann. Sc. Norm. Super. Pisa Cl. Sci. (5) \textbf{10} (2011),
  no.~2, 317--333. \MR{2856150}

\bibitem{S17}
Julian Scheuer, \emph{The inverse mean curvature flow in warped cylinders of
  non-positive radial curvature}, Adv. Math. \textbf{306} (2017), 1130--1163.
  \MR{3581327}

\bibitem{S19}
\bysame, \emph{Inverse curvature flows in {R}iemannian warped products}, J.
  Funct. Anal. \textbf{276} (2019), no.~4, 1097--1144. \MR{3906301}

\bibitem{U90}
John I.~E. Urbas, \emph{On the expansion of starshaped hypersurfaces by
  symmetric functions of their principal curvatures}, Math. Z. \textbf{205}
  (1990), no.~3, 355--372. \MR{1082861}

\bibitem{Z18}
Hengyu Zhou, \emph{Inverse mean curvature flows in warped product manifolds},
  J. Geom. Anal. \textbf{28} (2018), no.~2, 1749--1772. \MR{3790519}

\end{thebibliography}

\end{document}